\documentclass[12pt,notitlepage]{amsart}
\usepackage{latexsym,amsfonts,amssymb,amsmath,amsthm}
\usepackage{graphicx}
\usepackage{color}
\usepackage{mathrsfs}
\usepackage{amsmath, amsthm, amsfonts, amssymb, cite}
\usepackage{wrapfig}
\usepackage{stackrel}
\usepackage{color}
\usepackage{float}
\usepackage{enumitem}
\usepackage{mathtools}
\usepackage{multicol}
\usepackage[normalem]{ulem}

\usepackage[inner=1.0in,outer=1.0in,bottom=1.0in, top=1.0in]{geometry}

\theoremstyle{plain}		
	\newtheorem{mytheo}{Theorem} [section]

	\newtheorem{mycoro}[mytheo]{Corollary}
     \newtheorem{mylemma}[mytheo]{Lemma}
	\newtheorem{mydefi}[mytheo]{Definition}

\theoremstyle{remark}

\numberwithin{equation}{section}
\numberwithin{figure}{section}

\newcommand{\mz}{\ensuremath{\mathbb Z}}

\newcommand{\mh}{\ensuremath{\mathbb H}}

\newcommand{\mc}{\ensuremath{\mathbb C}}

\newcommand{\mymod}{\ensuremath{\negthickspace \negmedspace \pmod}}
\newcommand{\shortmod}{\ensuremath{\negthickspace \negthickspace \negthickspace \pmod}}

\DeclareMathOperator{\Hom}{Hom}

\begin{document}

\title{Dedekind sums arising from newform Eisenstein series}

\author{Tristie Stucker}
\email{stuc7464@vandals.uidaho.edu}
\address{University of Idaho \\ Moscow, ID 83844}

\author{Amy Vennos} 
\email{avennos3@gulls.salisbury.edu}
\address{Salisbury University \\
Salisbury, MD 21801}

\author{Matthew P. Young} 
\email{myoung@math.tamu.edu}
\address{Department of Mathematics\\
	  Texas A\&M University\\
	  College Station\\
	  TX 77843-3368\\
		U.S.A.}


 \thanks{
 This work was conducted in summer 2018 during an REU conducted at Texas A\&M University. The authors thank the Department of Mathematics at Texas A\&M and the NSF for supporting the REU.
In addition, this material is based upon work supported by the National Science Foundation under agreement No.\ DMS-170222 (M.Y.).  Any opinions, findings and conclusions or recommendations expressed in this material are those of the authors and do not necessarily reflect the views of the National Science Foundation.
 }

\begin{abstract}
For primitive non-trivial Dirichlet characters $\chi_1$ and $\chi_2$, we study the weight zero newform Eisenstein series $E_{\chi_1,\chi_2}(z,s)$ at $s=1$.  The holomorphic part of this function has a transformation rule that we express in finite terms as a generalized Dedekind sum.
This gives rise to the explicit construction (in finite terms) of elements of $H^1(\Gamma_0(N), \mc)$.
We also give a short proof of the reciprocity formula for this Dedekind sum.
\end{abstract}

\maketitle

\section{Introduction}
\subsection{Background and statement of result}
Let $\chi_1,\chi_2$ be primitive Dirichlet characters modulo $q_1,q_2$, respectively, with $\chi_1 \chi_2(-1)= 1$. 
The weight zero newform Eisenstein series attached to $\chi_1$ and $\chi_2$ is defined (initially) as
\begin{equation*}
E_{\chi_1,\chi_2}(z,s)=\frac{1}{2}\sum_{(c,d)=1}\frac{(q_2y)^s\chi_1(c)\chi_2(d)}{|cq_2z+d|^{2s}}, \qquad \text{Re}(s) > 1.
\end{equation*}
Here $E_{\chi_1,\chi_2}$ is an automorphic form on the congruence subgroup $\Gamma_0(q_1q_2)$ with central character $\psi = \chi_1 \overline{\chi_2}$. Precisely, for all $\gamma = 
\left( \begin{smallmatrix}
a& b\\
c& d
\end{smallmatrix} \right) \in \Gamma_0(q_1q_2)$,  
\begin{equation}
\label{eq:Echi1chi2transformation}
E_{\chi_1,\chi_2}(\gamma z,s)=\psi(\gamma)E_{\chi_1,\chi_2}(z,s),
\end{equation}
where $\psi(\gamma)=\psi(d)$. Moreover, $E_{\chi_1, \chi_2}$ is an eigenfunction of all the Hecke operators (see \eqref{eq:Echi1chi2HeckeOperator} below), which indicates why it is called a newform.  We refer to \cite{young} for the properties of the newform Eisenstein series used in this paper.

The classical Kronecker limit formula relates the constant term in the Laurent expansion of $E_{1,1}(z,s)$ at $s=1$ to $\log{\eta}$, where $\eta$ is the Dedekind $\eta$-function  
given by 
\begin{equation*}
\eta(z)=e^{\pi iz/12}\prod_{n=1}^{\infty}(1-e^{2\pi inz}) .
\end{equation*}
For 
$\gamma = 
\left( \begin{smallmatrix}
a& b\\
c& d
\end{smallmatrix} \right) \in \Gamma_0(1)$, 
with $c >0$,
$\log\eta$ obeys the transformation formula
\begin{equation*}\log{\eta(\gamma z)}= \log{ \eta(z)} +\pi i\Big(\frac{a+d}{12c}+s(-d,c)\Big)+\tfrac{1}{2}\log(-i(cz+d)),
\end{equation*}
where $s(h,k)$ is the classical Dedekind sum given by
\begin{equation*}
s(h,k)=\sum_{r=1}^{k-1}\frac{r}{k}\Big(\frac{hr}{k}-\Bigl\lfloor\frac{hr}{k}\Bigr\rfloor-\frac{1}{2}\Big). 
\end{equation*}
See \cite{apostol} for more background on the $\eta$-function and the classical Dedekind sums.


Consider the ``completed" Eisenstein series defined by
\begin{equation}
\label{eq:completedEisensteinSeriesDefinition}
 E_{\chi_1,\chi_2}^*(z,s)=\frac{(q_2/\pi)^s}{\tau(\chi_2)}\Gamma(s)L(2s,\chi_1\chi_2)E_{\chi_1,\chi_2}(z,s).
\end{equation}
Here $\tau$ denotes the Gauss sum given by 
$\tau(\chi)=\sum_{n \mymod{q}} \chi(n) e_q(n)$,
where $e_q(n) = e(n/q)$, $e(x) = \exp(2 \pi i x)$, and $\chi$ is a Dirichlet character modulo $q$.
The Fourier expansion for $E^*_{\chi_1,\chi_2}$ is conveniently stated in \cite{young} (see also \cite{huxley}). When $q_1,q_2 \neq 1$, the Fourier expansion simplifies as
\begin{equation} 
\label{fourier_expansion}
E_{\chi_1, \chi_2}^*(z,s) = 2\sqrt{y}\sum_{n \neq 0}\lambda_{\chi_1, \chi_2} (n,s) e(nx) K_{s-\frac{1}{2}} ( 2 \pi |n| y),
\end{equation}
 where $K_{\nu}$ is the $K$-Bessel function and
\begin{equation}
\label{eq:Echi1chi2FourierCoefficient}
 \lambda_{\chi_1, \chi_2} (n,s) = \chi_2 (\textnormal{sgn} (n) ) \sum_{ab=|n|}\chi_1 (a) \overline{\chi_2} (b) \left( \frac{b}{a} \right)^{s-\frac{1}{2}}.
\end{equation}
The Fourier expansion gives the analytic continuation of $E_{\chi_1, \chi_2}^*(z,s)$ to $s \in \mc$.  In particular, there is no pole at $s=1$, and \eqref{fourier_expansion} specializes as 
\begin{equation}\label{kronecker}
E^*_{\chi_1,\chi_2}(z,1)= f_{\chi_1,\chi_2}(z) + \chi_2 (-1)\overline{f}_{\overline{\chi_1}, \overline{\chi_2}} (z),
\end{equation} 
where 
\begin{equation}
\label{eq:fchi1chi2FourierExpansion}
 f_{\chi_1,\chi_2}(z)=\sum_{n =1}^{\infty} \frac{\lambda_{\chi_1, \chi_2}(n,1)}{\sqrt{n}}  e(nz),
\end{equation}
using $K_{1/2}(2 \pi y) = 2^{-1} y^{-1/2} \exp(-2 \pi y)$.
Because $E^*_{\chi_1,\chi_2}(z,s)$ has no pole at $s=1$, (\ref{kronecker}) is the analogue of the Kronecker limit formula and the function $f_{\chi_1,\chi_2}$ is the analogue of $\log\eta$. 

Define 
\begin{equation*}
\phi_{\chi_1,\chi_2}(\gamma, z)
=\phi_{\chi_1,\chi_2}(\gamma) 
= f_{\chi_1,\chi_2}(\gamma z)-\psi(\gamma)f_{\chi_1,\chi_2}(z),
\end{equation*}
for $\gamma \in \Gamma_0(q_1q_2)$ and $z \in \mh$;
in Lemma \ref{phiind} below, we show that $\phi_{\chi_1,\chi_2}$ is independent of $z$.

\begin{mydefi}
Let $\chi_1, \chi_2$ be primitive Dirichlet characters of conductors $q_1, q_2$, respectively, with $q_1, q_2 > 1$, and $\chi_1 \chi_2(-1) = 1$.  For $\gamma \in \Gamma_0(q_1 q_2)$, define the \emph{Dedekind sum} $S_{\chi_1, \chi_2}$ associated to the newform Eisenstein series $E_{\chi_1, \chi_2}$ by
\begin{equation}
\label{eq:DedekinSumDefinition}
 S_{\chi_1, \chi_2}(\gamma) = \frac{\tau (\overline{\chi_1})}{\pi i} \phi_{\chi_1, \chi_2}(\gamma).
\end{equation}
\end{mydefi}
Let $B_1$ denote the first Bernoulli function given by 
\begin{equation*}
B_1 (x) = \begin{cases} 
      x - \lfloor x \rfloor - \frac{1}{2} & \textnormal{if} \  x \in \mathbb{R} \backslash \mathbb{Z} \\
      0 & \textnormal{if} \ x \in \mathbb{Z}.
   \end{cases}
   \end{equation*}
The first main result in this paper is an evaluation of $S_{\chi_1, \chi_2}$ in finite terms: 
\begin{mytheo} 
\label{theorem}
Let $\chi_1, \chi_2$ be primitive Dirichlet characters of conductors $q_1, q_2$, respectively, with $q_1, q_2 > 1$, and $\chi_1 \chi_2(-1) = 1$.
Let $\gamma = 
\left( \begin{smallmatrix}
a& b\\
c& d
\end{smallmatrix} \right)
\in \Gamma_0(q_1q_2)$. For $c \geq 1$, then
\begin{equation}
\label{eq:phig1g2FiniteSum}
S_{\chi_1, \chi_2} (\gamma) =  \sum_{j \shortmod{c}} \sum_{n \shortmod{q_1}}\overline{\chi_2} (j)  \overline{\chi_1} (n) B_1\Big(\frac{j}{c}\Big) B_1 \left( \frac{n}{q_1}+\frac{aj}{c} \right). 
\end{equation}
\end{mytheo}
Our second main result gives a simple proof of the following reciprocity formula:
\begin{mytheo}
\label{thm:reciprocity}
For $\gamma = (\begin{smallmatrix} a & b \\ cq_1 q_2 & d \end{smallmatrix}) \in \Gamma_0(q_1 q_2)$, let $\gamma' = (\begin{smallmatrix} d & -c \\ -bq_1 q_2 & a \end{smallmatrix}) \in \Gamma_0(q_1 q_2)$.  If $\chi_1$ and $\chi_2$ are even, then
\begin{equation}
S_{\chi_1, \chi_2}(\gamma) = S_{\chi_2, \chi_1}(\gamma').
\end{equation}
If $\chi_1$ and $\chi_2$ are odd, then
\begin{equation}
S_{\chi_1, \chi_2}(\gamma) = -S_{\chi_2, \chi_1}(\gamma') + (1 - \psi(\gamma)) \frac{\tau(\overline{\chi_1}) \tau(\overline{\chi_2})}{(\pi i)^2} L(1, \chi_1) L(1, \chi_2).
\end{equation}
\end{mytheo}
The main step in the proof of Theorem \ref{thm:reciprocity} is to study the action of the Fricke involution $\omega_{q_1 q_2} = (\begin{smallmatrix} 0 & -1 \\ q_1 q_1 & 0 \end{smallmatrix})$.  Since $E_{\chi_1, \chi_2}$ is a pseudo-eigenvector of all the Atkin-Lehner operators (see \cite[Section 9]{young}), it seems plausible that an adaptation of the proof can give a family of reciprocity formulas, one for each Atkin-Lehner operator.

Many authors have investigated generalized Dedekind sums arising from various types of Eisenstein series.  Goldstein \cite{goldstein} studies the Eisenstein series attached to cusps for the principal congruence subgroup $\Gamma(N)$.
Nagasaka \cite{nagasaka} and Goldstein and Razar \cite{goldsteinrazar} investigate  functions essentially equivalent, in our notation, to $f_{\chi, \chi}$; they derive  the  transformation properties of $f_{\chi,\chi}$ (including the reciprocity formula) by relation  to the Mellin transform of the product of Dirichlet $L$-functions instead of via properties of Eisenstein series.  

The generalized Dedekind sums attached to pairs of Dirichlet characters have appeared in the literature in connection with certain Eisenstein-type series.  
Berndt \cite[Section 6]{berndt1975} defines generalized Dedekind sums which essentially correspond to the right hand side of 
\eqref{eq:phig1g2FiniteSum} when $q_1 = 1$ or $q_2 = 1$.  Berndt derives properties of his Dedekind sums using 
a different variant of Eisenstein series than what is used in this paper; Berndt's Eisenstein-type series have more complicated
transformation properties than $E_{\chi_1, \chi_2}$ (compare \cite[Theorem 2]{berndt1973} to \eqref{eq:Echi1chi2transformation}).
Many authors have studied generalized Dedekind sums, such as \cite{meyer} \cite{sekine} \cite{CCK} \cite{turkish}, based ultimately on Berndt's transformation formulas. 

Reciprocity formulas for
variants of $S_{\chi_1, \chi_2}$, with general pairs of characters $\chi_1, \chi_2$ have appeared in
\cite{turkish}.  However, it appears that Theorem \ref{thm:reciprocity} is new (e.g. \cite[Theorem 1]{turkish} excludes the case $p=1$ which would correspond to Theorem \ref{thm:reciprocity}).

In Section \ref{section:EichlerShimura} we connect $S_{\chi_1, \chi_2}$ to the Eisenstein component of the Eichler-Shimura isomorphism in weight $2$.

\subsection{Acknowledgements} The third author thanks Riad Masri and Ian Petrow for thoughtful comments.

\section{Basic properties of $S_{\chi_1,\chi_2}$}
\begin{mylemma}\label{phiind}
The function $\phi_{\chi_1,\chi_2}$ is independent of $z$.
\end{mylemma}

\begin{proof}
Since $E_{\chi_1, \chi_2}^* ( \gamma z,1) = \psi (\gamma) E_{\chi_1, \chi_2}^* (z,1) $ and $E^*_{\chi_1,\chi_2}(z,1)= f_{\chi_1,\chi_2}(z) + \chi_2 (-1)
\overline{f}_{\overline{\chi_1}, \overline{\chi_2}} (z)$, it immediately follows that
\begin{equation} \label{holoantiholo}
\phi_{\chi_1, \chi_2} (\gamma , z) = - \chi_2(-1)\overline{\phi}_{\overline{\chi_1},\overline{\chi_2}}(\gamma,z).
\end{equation}
Since $\phi_{\chi_1, \chi_2}$ is holomorphic and $\overline{\phi}_{\overline{\chi_1},\overline{\chi_2}}$ is antiholomorphic, $\phi_{\chi_1, \chi_2}$ must be constant in $z$.
\end{proof}

For later reference, we point out a symmetrized form for $\phi_{\chi_1, \chi_2}$ following from \eqref{holoantiholo}:
\begin{equation}\label{avg}
\phi_{\chi_1,\chi_2}(\gamma)=\tfrac{1}{2}(\phi_{\chi_1,\chi_2}(\gamma) -\chi_2(-1)\overline{\phi}_{\overline{\chi_1},\overline{\chi_2}}(\gamma)).
\end{equation}

\begin{mylemma}
\label{lemma:crossedhomomorphism}
Let $\gamma_1,\gamma_2 \in \Gamma_0(q_1q_2)$. Then 
\begin{equation}
S_{\chi_1,\chi_2}(\gamma_1\gamma_2)=S_{\chi_1,\chi_2}(\gamma_1)+\psi(\gamma_1)S_{\chi_1,\chi_2}(\gamma_2).
\end{equation}
\end{mylemma}

\noindent\textbf{Remarks.}
It is obvious from the definition that $S_{\chi_1, \chi_2}(\gamma) = 0$ if $\gamma = (\begin{smallmatrix} 1 & n \\ 0 & 1 \end{smallmatrix})$, for $n \in \mz$, and consequently 
$S_{\chi_1, \chi_2}(\gamma)$ only depends on the lower row of $\gamma$ (or, alternatively, the first column of $\gamma$).

Let $G = \Gamma_0(q_1 q_2)$ and $M= \mc$, and consider the action $\gamma.z$ of $G$ on $M$ given by 
$\gamma.z = \psi(\gamma) z$.  Note $G$ acts via automorphisms on $M$ (as a module).
With this notation, Lemma \ref{lemma:crossedhomomorphism} shows that $S_{\chi_1, \chi_2}$ is a $1$-cocycle (or a crossed homomorphism) for this group action of $G$ on $M$.  
Hence, $S_{\chi_1, \chi_2}$ gives rise to an element of $H^1(G,M)$.
In particular, if $\psi$ is trivial then $H^1(G,M) = \Hom(\Gamma_0(q_1 q_2), \mc)$ (i.e.,
$S_{\chi_1, \chi_2}$ is a group homomorphism).  Note also that $\psi$ is trivial on $\Gamma_1(q_1 q_2)$ so $S_{\chi_1, \chi_2}$ may always be viewed as an element of $\Hom(\Gamma_1(q_1 q_2), \mc)$.

\begin{proof}
Since $\psi$ is multiplicative, and by Lemma \ref{phiind},  we have
\begin{equation*}
\phi_{\chi_1,\chi_2}(\gamma_1\gamma_2)
= 
\underbrace{f_{\chi_1,\chi_2}(\gamma_1\gamma_2 z)- \psi(\gamma_1)f_{\chi_1,\chi_2}(\gamma_2 z)}_{\phi_{\chi_1, \chi_2}(\gamma_1)} 
+ \psi(\gamma_1) 
\underbrace{(f_{\chi_1,\chi_2}(\gamma_2 z) -\psi(\gamma_2)f_{\chi_1,\chi_2}(z))}_{\phi_{\chi_1,\chi_2}(\gamma_2)}.
\qedhere
\end{equation*} 
\end{proof}


 Let $T_n$ be the Hecke operator acting on weight $0$ periodic functions, with character $\chi$ (cf. \cite[(6.13)]{IwaniecTopics}), defined by
 \begin{equation*}
  (T_n f)(z) = \frac{1}{\sqrt{n}} \sum_{ad=n} \chi(a) 
  \sum_{b \shortmod{d}}
  f\Big(\frac{az+b}{d}\Big).
 \end{equation*}
It is easy to check that 
\begin{equation}
\label{eq:Echi1chi2HeckeOperator}
T_n E_{\chi_1, \chi_2}^*(z,s) = \lambda_{\chi_1, \chi_2}(n, s) E_{\chi_1, \chi_2}^*(z,s),
\end{equation}
for any $n \geq 1$.  We remark in passing that
\begin{equation}
 T_n f_{\chi_1, \chi_2} = \lambda_{\chi_1, \chi_2}(n,1) f_{\chi_1, \chi_2},
\end{equation}
 which
follows immediately from \eqref{kronecker}  and the fact that the Hecke operators preserve holomorphicity (and anti-holomorphicity).

\section{Proof of Theorem \ref{theorem}}
Our goal for the proof of Theorem \ref{theorem} is to use properties of $f_{\chi_1,\chi_2}$ in order to simplify $\phi_{\chi_1,\chi_2}$ and write it in finite terms. Our process loosely follows the methodology of Goldstein \cite{goldstein}. Let $\gamma = \left(\begin{smallmatrix}
a& b\\
c& d
\end{smallmatrix}\right) \in \Gamma_0(q_1q_2)$, with $c \geq 1$, and let $z= \frac{-d}{c}+\frac{i}{c^2u}$ for some $u > 0$. Then $\gamma z = \frac{a}{c}+iu$, 
and
\begin{equation*}
\phi_{\chi_1, \chi_2}(\gamma)= \lim_{u \rightarrow 0^+} \big[  f_{\chi_1,\chi_2}(\tfrac{a}{c}+iu)-\psi(\gamma)f_{\chi_1,\chi_2}(\tfrac{-d}{c}+\tfrac{i}{c^2u}) \big].
\end{equation*} 
From the Fourier expansion of $E_{\chi_1,\chi_2}^*$, it is clear that
$\lim_{u \rightarrow 0^+} f_{\chi_1,\chi_2}\left(\frac{-d}{c}+\frac{i}{c^2u}\right) = 0$. Thus,
\begin{equation}
\label{philimit}
\phi_{\chi_1, \chi_2}(\gamma)= \lim_{u \rightarrow 0^+} \ f_{\chi_1,\chi_2}(\tfrac{a}{c}+iu).
\end{equation}
This is the ``constant term'' in the Fourier expansion of $f_{\chi_1, \chi_2}$ around the cusp $a/c$.

To evaluate this limit, we begin by writing $f_{\chi_1,\chi_2}$ as
\begin{equation*}
f_{\chi_1,\chi_2}(z) = \sum_{k=1}^{\infty}\sum_{l=1}^{\infty}\frac{\chi_1(l)\overline{\chi_2}(k)}{l}e(klz).
\end{equation*}
Then 
\begin{equation}\label{fwiththeta}
f_{\chi_1,\chi_2}(z) =
\sum_{l=1}^{\infty}\frac{\chi_1(l)}{l} \theta_{\chi_2} (z,l), \quad \textnormal{where} \quad \theta_{\chi}(z,l):=\sum_{k=1}^{\infty}\overline{\chi}(k)e(klz).
\end{equation}

The following lemma will be used in several of the proofs below.
\begin{mylemma}\label{simplifying}
Let $\chi$ be a character of conductor $q$. Let $a,c,l \in \mathbb{Z}$ with $c \geq 1$, $c \equiv 0 \pmod{q}$, $(a,c)=1$, and $l \not\equiv 0 \pmod{\frac{c}{q}}$. Then $$\sum_{j \shortmod{c}} \overline{\chi}(j)
e_c(a lj) = 0.
$$
\end{mylemma}

\begin{proof}
Let $j=A+qB$ where $A$ runs modulo $q$ and $B$ runs modulo $c/q$. Then
\begin{equation*}\sum_{j \shortmod{c}} \overline{\chi} (j) 
e_c(alj)
 = \sum_{A \shortmod{q}}\overline{\chi}(A)
 e_c(alA)
 \sum_{B \shortmod{c/q}}
 e_{c/q}(alB)
 .\end{equation*} 
Since $\frac{c}{q} \nmid al$, the sum over $B$ vanishes.
\end{proof}

\begin{mylemma}\label{kappa}
Let $\chi$ be a character of conductor $q$. Let $a,c,l \in \mathbb{Z}$ with $c \geq 1$, $c \equiv 0 \pmod{q}$, $(a,c)=1$, and $l \not\equiv 0 \pmod{\frac{c}{q}}$. Then
\begin{equation*} 
\theta_{\chi}\left(\tfrac{a}{c}+iu, l\right) =\sum_{j=1}^{c-1} \overline{\chi} (j) 
e_c(alj)
\frac{x^j-1}{1-x^c}, \qquad \text{where} \qquad x= e(iul).
\end{equation*}
\end{mylemma}

\begin{proof}
We have
\begin{equation*}
\theta_{\chi}\left(\tfrac{a}{c}+iu,l\right)=\sum_{k=0}^{\infty}\overline{\chi}(k)
e_c(akl) x^k.
\end{equation*}
Now let $k=j+mc$ where $0 \leq j < c$ and $m$ runs over non-negative integers. Then
\begin{equation} 
\label{simplification}
\theta_{\chi} \left(\tfrac{a}{c} + iu,l \right)
=  \sum_{j=0}^{c-1}\overline{\chi} (j)
e_c(ajl)
 x^j
 \sum_{m=0}^{\infty} x^{mc} 
=\sum_{j=1}^{c-1} \overline{\chi}(j) e_c(ajl) \frac{x^j}{1-x^c}. 
\end{equation}
Using Lemma \ref{simplifying} and adding $0=\sum_{j=1}^{c-1} \overline{\chi}(j) e_c(ajl) \frac{-1}{1-x^c}$ to \eqref{simplification} completes the proof.
\end{proof}

\begin{mycoro}\label{corollary}
Under the same assumptions as Lemma \ref{kappa},
\begin{equation*}\lim_{u \rightarrow 0^+} \theta_{\chi}\left(\tfrac{a}{c}+iu, l\right) =-\sum_{j \shortmod{c}} \overline{\chi} (j)B_1\Big(\frac{j}{c}\Big) e_c(alj).
\end{equation*}
\end{mycoro}
\begin{proof}
As $u$ approaches 0, $x=e(iul)$ approaches 1, and 
$\lim_{x \rightarrow 1} \frac{x^j-1}{1-x^c} = \frac{-j}{c}.$
Thus,
\begin{equation*}
\lim_{u \rightarrow 0} \theta_{\chi}\left(\tfrac{a}{c}+iu, l\right) = \sum_{j=1}^{c-1} \frac{-j}{c} \overline{\chi} (j) e_c(alj) 
=-\sum_{j \shortmod{c}} \overline{\chi} (j) \left( \tfrac{j}{c} - \left\lfloor \tfrac{j}{c} \right\rfloor - \tfrac{1}{2} + \tfrac{1}{2} \right) e_c(alj).
\end{equation*}
Note $\overline{\chi}(j)\left(\frac{j}{c} - \left\lfloor \frac{j}{c} \right\rfloor - \frac{1}{2}\right)= \overline{\chi}(j)B_1\left(\frac{j}{c}\right)$, since
$\overline{\chi}(j)=0$ when $\frac{j}{c} \in \mathbb{Z}$, so 
using Lemma \ref{simplifying} again finishes the proof.
\end{proof}

\noindent\textbf{Remark.} 
We need a definition of the generalized Bernoulli function for a (primitive) Dirichlet character $\chi$ modulo $q$, which is stated in \cite[Definition 1]{berndt1975EulerMac}. One may easily unify Berndt's formulas as
\begin{equation}
\label{eq:B1chidef}
B_{1,\chi} (x)  = \frac{- \tau (\overline{\chi})}{2 \pi i} \sum_{\substack{l \in \mathbb{Z} \\ l \neq 0}} \frac{\chi(l)}{l} e_q(l x).
\end{equation} 

\begin{proof}[$\boldsymbol{\mathbf{Proof \ of \ Theorem \ \ref{theorem}}}$]

We apply (\ref{fwiththeta}) to (\ref{philimit}). Provided that we can interchange the limits (see Lemma \ref{limit} below),
\begin{equation}
\label{limitswap}
\phi_{\chi_1,\chi_2}(\gamma)
=\lim_{u \rightarrow 0^+} \sum_{l=1}^{\infty}\frac{\chi_1(l)}{l}\theta_{\chi_2}\left(\frac{a}{c}+iu,l \right) 
= \sum_{l=1}^{\infty}\frac{\chi_1(l)}{l}\lim_{u \rightarrow 0^+}\theta_{\chi_2}\left(\frac{a}{c}+iu,l \right).
\end{equation}
Then by Corollary \ref{corollary},
\begin{equation*}
\phi_{\chi_1,\chi_2}(\gamma)=-\sum_{l=1}^{\infty}\frac{\chi_1(l)}{l} \sum_{j \shortmod{c}} \overline{\chi_2}(j) B_1\Big(\frac{j}{c}\Big) e_c(a l j). 
\end{equation*} 
Applying \eqref{avg}, we obtain 
\begin{multline*}
\phi_{\chi_1,\chi_2}(\gamma) 
= -\frac{1}{2} \sum_{l=1}^{\infty}\frac{\chi_1(l)}{l} \sum_{j \shortmod{c}} \overline{\chi_2} (j) B_1\Big(\frac{j}{c}\Big) e_c(alj)  \\
 + \frac{\chi_2(-1)}{2}\sum_{l=1}^{\infty}\frac{\chi_1(l)}{l} \sum_{j \shortmod{c}} \overline{\chi_2} (j) B_1\Big(\frac{j}{c}\Big) e_c(-alj).
\end{multline*}
Changing variables $l \rightarrow -l$ and using $\chi_1 \chi_2(-1)=1$, this simplifies as
\begin{equation*}
\phi_{\chi_1,\chi_2}(\gamma) 
= -\frac{1} {2} \sum_{j \shortmod{c}}  \overline{\chi_2} (j) B_1\Big(\frac{j}{c}\Big) \sum_{l \neq 0} \frac{\chi_1 (l)}{l} e_c(alj).
\end{equation*}
Letting $c=c'q_1$ and substituting \eqref{eq:B1chidef}, we obtain
\begin{equation}
\label{eq:phichi1chi2gammaBernoullisWithChi}
\phi_{\chi_1,\chi_2}(\gamma) = \frac{\pi i }{ \tau (\overline{\chi_1})}\sum_{j \shortmod{c}} \overline{\chi_2}(j) B_1\Big(\frac{j}{c}\Big) B_{1, \chi_1} \Big( \frac{aj}{c'} \Big). 
\end{equation}
Next we use \cite[Theorem 3.1]{berndt1975EulerMac} which states
\begin{equation}
\label{eq:B1chidecomposition}
B_{1,\chi}(x)=\sum_{n=1}^{q-1}\overline{\chi}(n)B_1\left(\frac{x+n}{q}\right).
\end{equation} 
Substituting \eqref{eq:B1chidecomposition} into \eqref{eq:phichi1chi2gammaBernoullisWithChi} completes the proof.
%
\end{proof}

\begin{mylemma} \label{limit}
The interchange of limits in \eqref{limitswap}  is justified.
\end{mylemma}

\begin{proof}
Applying Lemma \ref{kappa} to the left hand side of \eqref{limitswap}, we have
\begin{equation*}
\phi_{\chi_1,\chi_2}(\gamma) = \lim_{u \rightarrow 0^+} \sum_{l=1}^{\infty}\frac{\chi_1(l)}{l} \sum_{j=0}^{c-1}\overline{\chi_2}(j)
e_c(alj)
\frac{x^j-1}{1-x^c}.
\end{equation*}
Let $R(x)=R_{j,c}(x)=\frac{x^j-1}{1-x^c}.$ Note that $R$ is a rational function (in $x$) with no poles on $0 \leq x \leq 1$, so it is smooth on this interval. 

Let $a_l=\chi_1(l)e_c(alj)$, $b_l=\frac{1}{l}R(e^{-2 \pi lu})$, and $S(N) = \sum_{l=1}^{N} a_{l}$. By Lemma \ref{simplifying}, $\sum_{l \mymod{c}} a_l  = 0$ (since we may assume $(j,q_2) = 1$ whence $j \not \equiv 0 \pmod{c/q_1}$), so $S(N)$ is bounded (independently of $u$, of course).  Therefore, by partial summation, 
 $\sum_{l=1}^{\infty} a_l b_l =  \sum_{l=1}^{\infty} S(l) (b_l - b_{l+1})$.
We claim $|b_l - b_{l+1}| = O(l^{-2})$, with an implied constant independent of $u$.  Given this claim, the Weierstrass $M$-test shows the sum converges uniformly in $u$ which justifies the interchange of limits.

Now we show the claim.
We have
\begin{equation}
 |b_{l+1}-b_{l}|= \frac{1}{l} \Big| R(e^{-2 \pi (l+1) u})-R(e^{-2 \pi l u}) - \frac{R(e^{-2 \pi (l+1)u})}{ l+1} \Big|.
\end{equation}
Here $\frac{|R(e^{-2 \pi (l+1)u})|}{ l+1} \leq \frac{C_1}{l}$ for some constant $C_1$ independent of $l$ and $u$.
By the mean value theorem, 
\begin{equation*}
R(e^{-2 \pi l u}) - R(e^{-2 \pi (l+1)u})  =(e^{-2 \pi lu} - e^{-2 \pi(l+1)u})R'(t)
\end{equation*}
for some $t \in [0,1]$.  Since $R(t)$ is smooth on $[0,1]$, then $|R'(t)| \leq C_2$ for some constant $C_2$ independent of $l$ and $u$. Additionally, 
\begin{equation*}
e^{-2 \pi lu} - e^{-2 \pi(l+1)u}= e^{-2 \pi l u}(1-e^{-2 \pi u}) \leq \frac{C_3}{l} ul e^{-2 \pi lu} \leq \frac{C_4}{l},
\end{equation*} 
for some constants $C_3$, $C_4$,
since $xe^{-x}$ is bounded for $0 \leq x < \infty$.
Putting everything together proves the claim.
\end{proof}

\section{Proof of Theorem \ref{thm:reciprocity}}
Let $\omega = \omega_{q_1 q_2} = (\begin{smallmatrix} 0 & -1 \\ q_1 q_1 & 0 \end{smallmatrix})$ be the Fricke involution.  An easy calculation shows that if $\gamma = (\begin{smallmatrix} a & b \\ cq_1 q_2 & d \end{smallmatrix}) \in \Gamma_0(q_1 q_2)$,  then 
\begin{equation}
\label{eq:omegagammareversalformula}
 \omega \gamma = \gamma' \omega,
\end{equation}
where $\gamma' = (\begin{smallmatrix} d & -c \\ -bq_1 q_2 & a \end{smallmatrix}) \in \Gamma_0(q_1 q_2)$.  Note the map $\gamma \rightarrow \gamma'$ is an involution.
The newform Eisenstein series is a generalized eigenfunction of the Fricke involution, precisely it satisfies (see \cite[Section 9.2]{young})
\begin{equation*}
 E_{\chi_1, \chi_2}(\omega z, s) = \chi_2(-1) E_{\chi_2, \chi_1}(z,s).
\end{equation*}
For the completed Eisenstein series, using \eqref{eq:completedEisensteinSeriesDefinition} we deduce
\begin{equation*}
 E_{\chi_1, \chi_2}^*(\omega z, 1)
  =  \delta_{\chi_1, \chi_2} E_{\chi_2, \chi_1}^*(z, 1),
  \quad \text{where}
  \quad \delta_{\chi_1, \chi_2} = \chi_2(-1) \frac{\tau(\chi_1) q_2}{\tau(\chi_2)q_1}.
\end{equation*}

Define
 $\phi_{\chi_1, \chi_2}(\omega) = f_{\chi_1, \chi_2}(wz) - \delta_{\chi_1, \chi_2} f_{\chi_2, \chi_1}(z)$,
and similarly define
 \begin{equation}
 S_{\chi_1, \chi_2}(\omega) = \frac{\tau(\overline{\chi_1})}{\pi i} \phi_{\chi_1, \chi_2}(\omega).
 \end{equation}
An easy modification of the proof of Lemma \ref{phiind} shows that $\phi_{\chi_1, \chi_2}(\omega)$ is independent of $z$ (justifying the notation).  
\begin{mylemma}
\label{lemma:DedekindSumEvaluationOnFricke}
 Let $\chi_1, \chi_2$ be primitive Dirichlet characters of conductors $q_1, q_2$, respectively, with $q_1, q_2 > 1$, and $\chi_1 \chi_2(-1) = 1$.  Then 
 \begin{equation}
  S_{\chi_1, \chi_2}(\omega) = 
  \begin{cases}
  \frac{\tau(\overline{\chi_1}) \tau(\overline{\chi_2})}{(\pi i)^2} L(1,\chi_1) L(1, \chi_2), \qquad &\chi_1(-1) = \chi_2(-1) = -1, \\
  0, \qquad &\chi_1(-1) = \chi_2(-1) = 1.
  \end{cases}
 \end{equation}
\end{mylemma}
\begin{proof}
 The ideas are similar to the proof of Theorem \ref{thm:reciprocity}, so we will be brief.  We have
 $\phi_{\chi_1,\chi_2}(\omega) = \lim_{u \rightarrow 0^{+}} f_{\chi_1, \chi_2}(iu)$.  Then following the idea of proof in Lemma \ref{kappa}, we have
 \begin{equation*}
f_{\chi_1, \chi_2}(iu) =  \sum_{\ell=1}^{\infty} \frac{\chi_1(\ell)}{\ell}
\sum_{0 \leq j < q_2} \overline{\chi_2}(j) 
\frac{x^j-1}{1-x^{q_2}}, \quad x = e(\ell i u).
 \end{equation*}
Letting $u \rightarrow 0^{+}$ (using a variant on Lemma \ref{limit} to change the limits) gives
\begin{equation*}
 \phi_{\chi_1, \chi_2}(\omega) = -
 L(1, \chi_1)
 B_{1,\chi_2}(0),
\end{equation*}
using \eqref{eq:B1chidecomposition}.  Finally, we use \eqref{eq:B1chidef} to complete the proof.
\end{proof}

Now we calculate $f_{\chi_1, \chi_2}(\omega \gamma z) - \delta_{\chi_1, \chi_2} \overline{\psi}(\gamma) f_{\chi_2, \chi_1}(z)$ in two ways.  One expression is 
\begin{equation*}
 \underbrace{f_{\chi_1, \chi_2}(\omega \gamma z) 
 -\delta_{\chi_1, \chi_2} f_{\chi_2, \chi_1}(\gamma z)}_{\phi_{\chi_1, \chi_2}(\omega)}
 + \delta_{\chi_1, \chi_2} 
 [\underbrace{f_{\chi_2, \chi_1}(\gamma z) - 
 \overline{\psi}(\gamma) f_{\chi_2, \chi_1}(z)}_{\phi_{\chi_2, \chi_1}(\gamma)}].
\end{equation*}
Alternatively, using \eqref{eq:omegagammareversalformula}, it equals
\begin{equation*}
 \underbrace{f_{\chi_1, \chi_2}(\gamma' \omega z) - \psi(\gamma') f_{\chi_1, \chi_2}(\omega z)}_{\phi_{\chi_1, \chi_2}(\gamma')} 
 +
 \overline{\psi}(\gamma)
 [\underbrace{f_{\chi_1, \chi_2}(\omega z) - \delta_{\chi_1, \chi_2} f_{\chi_1, \chi_2}(z)}_{\phi_{\chi_1, \chi_2}(\omega)}],
\end{equation*}
where we have used $\psi(\gamma') = \psi(a) = \overline{\psi}(d) = \overline{\psi}(\gamma)$.  Equating the two expressions, we derive
\begin{equation*}
 \phi_{\chi_1, \chi_2}(\gamma') - \delta_{\chi_1, \chi_2} \phi_{\chi_2, \chi_1}(\gamma) = (1 - \overline{\psi}(\gamma)) \phi_{\chi_1, \chi_2}(\omega).
\end{equation*}
Converting the notation using \eqref{eq:DedekinSumDefinition}, 
and using
$ \delta_{\chi_1, \chi_2}\frac{\tau(\overline{\chi_1})}{\tau(\overline{\chi_2})} = \chi_1(-1)$,
we derive
\begin{equation*}
 S_{\chi_1, \chi_2}(\gamma') - \chi_1(-1)  S_{\chi_2, \chi_1}(\gamma) = (1 - \overline{\psi}(\gamma)) S_{\chi_1, \chi_2}(\omega).
\end{equation*}
Using Lemma \ref{lemma:DedekindSumEvaluationOnFricke} 
and switching the roles of $\gamma$ and $\gamma'$
completes the proof of Theorem \ref{thm:reciprocity}.

%
%

\section{Remarks on the Eichler-Shimura isomorphism}
\label{section:EichlerShimura}
Let $E_{2,\chi_1, \chi_2}(z)$ be the holomorphic weight $2$ Eisenstein series attached to the primitive non-trivial characters $\chi_1, \chi_2$, defined by (using the notation \eqref{eq:Echi1chi2FourierCoefficient})
\begin{equation*}
 E_{2,\chi_1, \chi_2}(z) = 2 \sum_{n=1}^{\infty} n^{1/2} \lambda_{\chi_1, \chi_2}(n,1) q^n, \qquad q = e^{2\pi i z}.
\end{equation*}
See \cite[Section 4.6]{DiamondShurman} for more details.  The Eichler-Shimura map applied to $E_{2,\chi_1, \chi_2}$ is defined by
\begin{equation*}
\gamma \mapsto \int_{\infty}^{\gamma(\infty)} E_{2,\chi_1, \chi_2}(z) dz,
\end{equation*}
for $\gamma \in \Gamma_0(q_1 q_2)$.  By direct calculation with \eqref{eq:fchi1chi2FourierExpansion}, we have
\begin{equation*}
 \frac{d}{dz} \frac{1}{\pi i} f_{\chi_1, \chi_2}(z) = E_{2,\chi_1, \chi_2}(z).
\end{equation*}
Therefore the Eichler-Shimura map applied to $E_{2,\chi_1, \chi_2}$ is precisely $\tau(\overline{\chi_1})S_{\chi_1, \chi_2}$.

\end{document}